%% file: main.tex
\documentclass[12pt, oneside]{amsart}
\include{Preambolo}

\title[]{Regularity of conformal structures \\ on closed 3-manifolds}

\author{Rodrigo Avalos}
\address{Eberhard Karls Universit\"at T\"ubingen, Fachbereich Mathematik, Auf der Morgenstelle 10, 72076 T\"{u}bingen, Germany}

\email{rodrigo.avalos@mnf.uni-tuebingen.de}
\email{andoni.royo-abrego@uni-tuebingen.de}

\author{Albachiara Cogo}
\address{Centro di Ricerca Matematica Ennio de Giorgi, Scuola Normale Superiore di Pisa, Piazza dei Cavalieri 3, 56126, Pisa, Italy}
\email{albachiara.cogo@sns.it}

\author{Andoni Royo Abrego}

\begin{document}

\begin{abstract}
    It is well known in Riemannian geometry that the metric components have the best regularity in harmonic coordinates. These can be used to characterize the most regular element in the isometry class of a rough Riemannian metric. In this work, we study the conformal analogue problem on closed 3-manifolds: given a Riemannian metric $g$ of class $W^{2,q}$ with $q > 3$, we characterize when a more regular representative exists in its conformal class. We highlight a deep link to the Yamabe problem for rough metrics and present some immediate applications to conformally flat, static and Einstein manifolds.
\end{abstract}

\maketitle

\pagestyle{plain}

\section[{\textbf{Introduction and main results}}]{Introduction and main results}

A classic problem in geometric analysis is that of finding a good \textit{gauge} to study a particular problem. Since most equations of interest are geometric and thus intrinsically defined, there is, in principle, no canonical coordinate system in which to express them explicitly. On the other hand, the analytical techniques used to prove results such as the existence of solutions often depend on finding a suitable choice of coordinates. Well-known examples of this are \cite{ChoquetBruhat} for the vacuum Einstein field equations, \cite{Uhlenbeck} for Yang--Mills connections or \cite{DeTurckRF} for Ricci flow. In this work, we are concerned with finding smooth representatives of a Riemannian metric of rough regularity among all the elements in its isometry and conformal classes.

Let $M$ be a smooth, closed 3-manifold and consider a Riemannian metric $g \in W^{k,q}(M)$ for $k \in \nN$, $k \geq 2$ and $q > 3$. Let $\mathscr D^{l,p}(M)$ denote the space of all $W^{l,p}$ diffeomorphisms from $M$ to $M$  for $l \in \nN$ and $1 < p < \infty$. We then define
\begin{equation}
    \label{eq: conformal class def}
    \llbracket\,g\,\rrbracket_{W^{k,q}} \coloneqq \Bigl\{\Phi^*\bigl(u^4g\bigr) \; : \; u \in W^{k,q}(M) , \; u > 0, \; \Phi \in \mathscr D^{k+1,q}(M)\Bigr\} \,.
\end{equation}
Such equivalence classes are elements of the moduli space of $W^{k,q}$ Riemannian metrics on $M$ modulo the action of the diffeomorphism and conformal groups of the corresponding suitable regularity. It follows from Sobolev multiplication properties and coordinate transformation rules for tensors that \eqref{eq: conformal class def} is a well-defined equivalence class --see \cref{lemma_Adams}--.

There are instances where $\llbracket\,g\,\rrbracket_{W^{k,q}}$ is expected to have smooth representatives: for example, if we pull-back a Riemannian metric $g \in C^\infty(M)$ by a diffeomorphism in $\mathscr D^{k+1,q}(M)$, or multiply it by a $W^{k,q}(M)$ conformal factor, we certainly obtain a rough Riemannian metric, but the class $\llbracket\,g\,\rrbracket_{W^{k,q}}$ still contains a $C^\infty(M)$ metric. It is then a natural question to ask whether there exist conformal conditions for a rough metric $g$ such that $\llbracket\,g\,\rrbracket_{W^{k,q}}$ admits a more regular representative. Let $\C_g$ denote the Cotton tensor of $g$ --see \eqref{eq: Cotton def} for the definition--. The main result of this paper is the following:

\begin{restatable}[\textbf{Regularity of conformal classes}]{thm}{thmA}
    \label{thm: main}
    Let $M$ be an orientable, smooth, closed 3-manifold and consider a Riemannian metric $g \in W^{k,q}(M)$ for $k \geq 2$ and $q > 3$. Suppose $\C_g \in W^{l,q}(M)$ for some $l \in \nN_0$, $l \leq k$. Then, there exists a metric $\tilde g \in \llbracket\,g\,\rrbracket_{W^{k,q}}$ of constant scalar curvature such that $\tilde g \in W^{l+3,q}(M)$.
\end{restatable}
We recall that the Cotton tensor is third order in the metric and conformally invariant in dimension $n=3$. In particular, if there exists
$\tilde g \in \llbracket\,g\,\rrbracket_{W^{k,q}}$ such that $\tilde g \in W^{l+3,q}(M)$, then $\C_{\tilde g} \in W^{l,q}(M)$  and  
$\C_g = \Phi^*(\C_{\tilde g}) \in W^{l,q}(M)$, due to the coordinate transformation rule --see \cref{lemma_Adams}-- and the regularity of the diffeomorphism $\Phi \in \mathscr{D}^{k+1, q}(M)$ with $l \leq k$.
In other words, our result completely characterises the existence of more regular metrics in $\llbracket\,g\,\rrbracket_{W^{k,q}}$ within the regularity range stated in the theorem. 

Notice that, even if $\C_g$ was assumed to be in $C^\infty(M)$, one does not generally expect to find metrics in $\llbracket\,g\,\rrbracket_{W^{k,q}}$ with regularity exceeding $W^{k+3,q}(M)$. This limitation stems from the regularity of harmonic coordinates employed in the proof of \cref{thm: main} --see \cref{rmk: Cotton improv obstruction} and \cref{rmk: improvement obstruction} for details--. On the other hand, higher regularity of the Cotton tensor often emerges naturally in the presence of underlying geometric PDEs. Notable examples include Einstein metrics, Cotton-flat metrics, Cotton-parallel metrics, and static manifolds. In these cases, the geometric constraints ensure that it is actually possible to find $C^\infty(M)$ metrics in $\llbracket\,g\,\rrbracket_{W^{k,q}}$. To showcase this, we prove the following result as an application of \cref{thm: main}: 
\begin{corA}[\textbf{Regularity of Cotton flat metrics}]
    \label{thm: Cotton flat intro}
    Let $M$ be an orientable, smooth, closed $3$-manifold and consider a Riemannian metric $g \in W^{2,q}(M)$ with $q > 3$. Suppose that $\C_g \equiv 0$. Then there exists a metric $\tilde g \in \llbracket\,g\,\rrbracket_{W^{2,q}}$ such that $\tilde g \in C^\infty(M)$ and is locally conformal flat.
\end{corA}
\cref{thm: Cotton flat intro} is also the conformal counterpart of a result by M. Taylor \cite{Taylor_ConfFlat}, where he extends to rough metrics the classic assertion that the vanishing of the Riemann tensor implies the manifold is locally isometric to $\nE^n$ --see \cref{subsection: Conformal flatness}--.

We now discuss the arguments and related results that underpin the proof of \cref{thm: main}. The central idea is that the Ricci tensor satisfies an elliptic equation of the form\footnote{The notation $A*B$ stands for some linear combination of the components of $A$, $B$ and the metric $g$.} 
\begin{equation}
    \label{eq: Cotton PDE intro}
    \Delta_g \Ric_g = \Div_g\C_g + \Rm_g * \Ric_g + \nabla^2 \R_g \,,
\end{equation}
where $\Rm_g$, $\Ric_g$ and $\R_g$ denote the Riemann, the Ricci and scalar curvature of $g$, respectively. If the Cotton tensor, which is a conformal invariant, is more regular by hypothesis, then the regularity of $\Delta_g \Ric_g$ is improved, provided that the scalar curvature of $g$ is also more regular. We thereby first move to a conformal metric of constant scalar curvature, which existence is ensured by the recent resolution of the Yamabe problem in the class of metrics of \cref{thm: main} by the authors \cite[Theorem A]{AvalosCogoRoyo}. Once we fix this conformal gauge, the proof of \cref{thm: main} reduces to establishing the following result, which can be regarded as a regularity statement for rough Yamabe metrics:
\begin{restatable}[\textbf{Regularity of constant scalar curvature metrics}]{thm}{thmCottonconstantscalar}
    \label{thm:Cottonconstantscalar}
    Let $M$ be a smooth, closed $n$-manifold and suppose that $g \in W^{k,q}(M)$ with $k \geq 2$ and $q > n$ has constant scalar curvature. If $\C_g \in W^{l,q}(M)$ for some $l \leq k$, then there exists a $W^{k+1,q}$ diffeomorphism $\Phi : M \to M$ such that $\Phi^*g \in W^{l+3,q}(M)$.  
\end{restatable}
Notice that the orientability assumption in \cref{thm: main} and \cref{thm: Cotton flat intro} is absent in \cref{thm:Cottonconstantscalar}. This hypothesis is only used to ensure the applicability of the positive mass theorem of \cite{LeeLeF} in the resolution of the Yamabe problem for such metrics, but most likely is not needed, as commented in \cite{AvalosCogoRoyo}.

The proof of \cref{thm:Cottonconstantscalar} consists in exploiting \eqref{eq: Cotton PDE intro} with $\nabla^2 \R_g  =0$ by applying (non-standard) elliptic regularity theory to show that $\Ric_g$ is more regular. At this point, one needs to find an isometry gauge in which the metric is more regular. This is precisely the content of the following result:
\begin{restatable}[\textbf{Regularity of isometry classes}]{thm}{thmD}
    \label{thm: main Ricci}
    Let $M$ be a smooth, closed $n$-manifold and consider a Riemannian metric $g \in W^{k,q}(M)$ for $k \geq 2$ and $q > n$. Suppose that $\Ric_g \in W^{l,q}(M)$ for some $l \leq k$. Then, there exists a $W^{k+1,q}$ diffeomorphism $\Phi : M \to M$ such that $\Phi^*g \in W^{l+2,q}(M)$.
\end{restatable}
Also of independent interest, \cref{thm: main Ricci} is a global version for Sobolev metrics of the well-known observation of Sabitov--Shefel \cite{Sabitov-Shefel} and DeTurck--Kazdan \cite{DeTurck_Kazdan} that the regularity of the metric components in harmonic coordinates have improved regularity, provided that the components of the Ricci tensor are more regular than expected. The global diffeomorphism $\Phi$ is constructed using harmonic coordinates --see \cref{prop: harmonic atlas}-- combined with a careful application of results by H. Whitney \cite{Whitney}. As it happens with harmonic coordinates, the regularity of the diffeomorphism $\Phi$ is only $W^{k+1,q}$, while the metric $\Phi^*g$ is surprisingly better.

Additionally, we present two immediate applications of \cref{thm:Cottonconstantscalar} addressing examples of metrics with constant scalar curvature and
improved Cotton regularity. The first is a global and refined version of a result by J. Corvino \cite{Corvino_static} on static systems --see \cref{Static vacuum systems} for definitions--: 
\begin{corB}[\textbf{Regularity of static systems}]\label{coroll: regularity static systems}
     Let $(M, g, f)$ be a static system with $(g,f) \in W^{k,q}_{loc}(M) \times W^{k,q}_{loc}(M)$ for some $k \geq 2$ and $q > n$. Then, there exists a $W^{k+1,q}_{loc}$ diffeomorphism $\Phi : M \to M$ such that $(\Phi^*g,\Phi^*f) \in C^{\infty}_{loc}(M) \times C^{\infty}_{loc}(M)$.
\end{corB}
The second one is a global version for Sobolev metrics of a classical regularity theorem for Einstein manifolds due to DeTurck--Kazdan \cite[Theorem 5.2]{DeTurck_Kazdan}:
\begin{corB}[\textbf{Regularity of Einstein metrics}]\label{coroll: regularity Einstein}
    Let $M$ be a smooth $n$-manifold and let $g \in W^{k,q}_{loc}(M)$ be an Einstein metric on $M$ with $k \geq 2$ and $q > \tfrac{n}{2}$. Then there exists a $W^{k+1,q}_{loc}$ diffeomorphism $\Phi: M \to M $ such that $\Phi^* g \in C^{\infty}_{loc}(M)$.
\end{corB}
We remark that in \cref{coroll: regularity static systems} and \cref{coroll: regularity Einstein} we do not assume compactness of $M$, nor a structure at infinity; the $loc$ subscript can be removed if $M$ is compact. As in \cref{thm: Cotton flat intro}, an underlying geometric PDE allows avoiding the obstructions outlined in \cref{rmk: improvement obstruction} and improving regularity up to $C^{\infty}$.

\subsection{Outline of the paper}
In \cref{section: Preliminaries} we introduce some fundamental tools needed along the paper and collect the main elliptic regularity theorem we will use. In \cref{section: Regularity of metric structures} we study the regularity of isometry classes for rough metrics and prove \cref{thm: main Ricci}. \cref{section: Regularity of conformal structures} is devoted to conformal classes of rough metrics and the proof of \cref{thm: main}. Finally, in \cref{section: Applications} we present some applications of the main theorem, proving \cref{thm: Cotton flat intro}, \cref{coroll: regularity static systems} and \cref{coroll: regularity Einstein}.

\vspace{1cm}

\section[{\textbf{Preliminaries}}]{Preliminaries}
\label{section: Preliminaries}
\subsection{Sobolev diffeomorphisms and differentiable structures}
When dealing with Riemannian metrics of low regularity, it is convenient to work with a \textit{local} definition of Sobolev spaces. Namely, we say that a tensor field $\mathbf{u}$ is of class $W^{k,p}(M)$, if its components in any coordinate chart $(V,\varphi)$ of $M$ belong to $W^{k,p}_{loc}(\varphi(V))$. These spaces satisfy the usual embedding, density and multiplication properties on compact manifolds. We refer the reader to our previous work \cite[Section 2.1]{AvalosCogoRoyo} for a detailed discussion of these spaces and their relation to other equivalent definitions. In particular, we will make extensive use of the Sobolev multiplication properties in \cite[Appendix A]{AvalosCogoRoyo} --see \cite[Chapter 9]{PalaisBook} and \cite[Chapter VI]{ChoquetDeWitt} for classical references--.

In this paper, an important aspect of tensors fields of Sobolev regularity is the way they transform under rough change of coordinates; for instance, we will make use of harmonic coordinates induced by Riemannian metrics of low regularity, which are naturally not smooth. Hence, we recall in the following fundamental key result:
\begin{proposition}
    \label{lemma_Adams}
     Fix an integer $k \geq 2$ and real numbers $q > \tfrac{n}{2}$ and $1 \leq p \leq q$. Let $\phi: \Omega \to \Omega'$ be a $C^1$-diffeomorphism between two bounded domains in $\nR^n$ and let $\mathbf u$ be a tensor field in $\Omega$.
     \begin{enumerate}
         \item [(i)] If $\phi \in C^{k+1}(\Omega,\Omega')$ and $\mathbf{u} \in W^{k,p}(\Omega)$, then $\phi_* \mathbf{u} \in W^{k,p}_{loc}(\Omega')$. Moreover, if $\mathbf u$ is a scalar field, then $\mathbf{u} \circ \phi^{-1} \in W^{k+1,p}_{loc}(\Omega')$
         
         \item [(ii)] If $\phi \in W^{k+1,q}(\Omega,\Omega')$ and $\mathbf{u} \in W^{k,p}(\Omega)$, then $\phi_*\mathbf{u} \in W^{k,p}_{loc}(\Omega')$. Moreover, if $\mathbf u$ is a scalar field, then $\mathbf{u} \circ \phi^{-1} \in W^{k+1,p}_{loc}(\Omega')$
     \end{enumerate}
\end{proposition}
\begin{proof}
    If $\phi \in C^{k+1}(\Omega,\Omega')$, then $\phi^{-1} \in C^{k+1}_{loc}(\Omega',\Omega)$ by the inverse function theorem and assertion (\textit{i}) follows from the transformation rule for tensors together with \cite[Theorem 3.41]{AdamsFournierBook}. If $\phi \in W^{k+1,q}(\Omega,\Omega')$, one can show that $\phi^{-1} \in W^{k+1,q}_{loc}(\Omega',\Omega)$ as well (\cite[Theorem B.1]{AvalosCogoRoyo}) and a similar argument yields (\textit{ii}). See \cite[Lemma 2.1 - Lemma 2.2]{AvalosCogoRoyo} for the details.
\end{proof}
It is direct consequence \cref{lemma_Adams} and our definition of Sobolev spaces that if $\Phi : M \to M'$ is a $W^{k+1,q}$ diffeomorphism between closed manifolds and $\mathbf u \in W^{k,p}(M)$ is a tensor field in $M$ for some $1 \leq p \leq q$, then $\Phi_*\mathbf u \in W^{k,p}(M')$. In particular, if $M'$ is the topological manifold $M$ equipped with a distinct $C^\infty$ differential structure, which is only $W^{k+1,q}$ compatible to the original one, then $\mathbf u \in W^{k,p}(M)$ if and only if $\mathbf u \in W^{k,p}(M')$. In order to avoid working with non $C^\infty$ differential structures, we recall the following classical theorem attributed to H. Whitney (see \cite[Theorem 2.9]{Hirsch} for a modern statement and detailed proof):
\begin{theorem}
\label{thm: Withney}
    Let $M$ be a $C^k$ manifold. If $k \geq 1$, then there exists a $C^\infty$ differential structure on $M$, which is $C^k$ compatible with the original $C^k$ differential structure of $M$ and unique up to $C^\infty$ diffeomorphisms. 
\end{theorem}

\subsection{Harmonic coordinates and atlases}
Before introducing harmonic coordinates and atlases, we collect the following interior elliptic regularity for the Laplace--Beltrami operator $\Delta_g \coloneqq \Div_g \nabla = \tr_g\nabla^2$ for rough metrics. More general statements and detailed proofs can be found in \cite[Section 3]{AvalosCogoRoyo}.
\begin{theorem}
    \label{thm: Laplacian_regularity}
    Let $\Omega \subset \nR^n$
    be an open, bounded domain with smooth boundary and consider a Riemannian metric $g \in W^{k,q}(\Omega)$ with $k \geq 2$ and $q > n$. 
    \begin{enumerate}
        \item [(i)] If $u \in L^q_{loc}(\Omega)$ and $\Delta_g u \in W^{-1, q}_{loc}(\Omega)$, then $u \in W^{1, q}_{loc}(\Omega)$.
    
        \item [(ii)] If $u \in W^{k-1, q}_{loc}(\Omega)$ and $\Delta_g u \in W^{k-2, q}_{loc}(\Omega)$, then $u \in W^{k, q}_{loc}(\Omega)$.
    \end{enumerate}
\end{theorem}
\begin{proof}
    Since $q > n \geq 3$, we have that $L^q_{loc}(\Omega) \hookrightarrow L^{q'}_{loc}(\Omega)$ and assertion (\textit{i}) reduces to the $p = q$ case of \cite[Theorem 3.2-(\textit{i})]{AvalosCogoRoyo}. Similarly, noticing that $W^{k-1,q}(\Omega) \hookrightarrow W^{k-2,q'}(\Omega)$, the assertion (\textit{ii}) reduces to the $p = q$ case of \cite[Theorem 3.3]{AvalosCogoRoyo}.
\end{proof}

Among other important applications, \cref{thm: Laplacian_regularity} allows us to construct harmonic coordinates and \textit{harmonic atlases} using a rough Riemannian metric:
\begin{proposition}
    \label{prop: harmonic atlas} 
    Let $M$ be a smooth manifold of dimension $n \geq 3$ and consider a Riemannian metric $g \in W^{k,q}(M)$ for $k \geq 2$ and $q > n$.
    There exists a collection of charts $\mathcal A_H = \{(V_\beta, \varphi_\beta)\}_{\beta \in B}$ covering $M$ with the following properties.
    \begin{enumerate}
        \item [(i)] The coordinates $\{ x^i_{\beta}\}_{i =1}^n$ induced by any $\varphi_\beta$ are harmonic, that is
        \begin{equation*}
            \Delta_g x^i_{\beta} = 0\,.
        \end{equation*}
        \item [(ii)] For any $\beta \in B$, there holds $\varphi_{\beta} \in W^{k+1, q}(V_{\beta})$. Namely, each chart in $\mathcal A_H$ is $W^{k+1,q}$ compatible with the $C^{\infty}$ differential structure of $M$.
        \item[(iii)] $\mathcal A_H$ forms a $W^{k+2, q}$ atlas. Moreover, if the metric components in harmonic coordinates $\{ x^i_{\beta}\}_{i =1}^n$ are in $W^{l,q}(\varphi_\beta(V_\beta))$ for all $\beta \in B$ and some $l > k$, then $\mathcal A_H$ is a $W^{l+2, q}$ atlas. 
    \end{enumerate}
\end{proposition}
\begin{proof}
    The existence of a harmonic chart around any point in $M$ which is $W^{k+1, q}$ compatible with the original $C^{\infty}$ differentiable structure was proven in \cite[Proposition 3.13]{AvalosCogoRoyo}. The statements (\emph{i}) and (\emph{ii}) are then a consequence.

    To prove (\emph{iii}), we need to examine the regularity of the transition maps between any two harmonic charts. Let us fix two arbitrary harmonic charts $(V_1, \varphi_1)$ and $(V_2, \varphi_2)$ with nonempty intersection and let $\{x^i\}_{i=1}^n$ and $\{y^i\}_{i=1}^n$ be the coordinates induced by $\varphi_1$ and $\varphi_2$, respectively. Consider also a chart $(\hat V, \hat \varphi)$ of the original $C^{\infty}$ differential structure such that $(V_1 \cap V_2) \subset \hat V$ (or finitely many of them covering $(V_1 \cap V_2)$, if necessary) . We then write the transition map
    \begin{equation*}
        \varphi_1 \circ \varphi_2^{-1} = (\varphi_1 \circ \hat \varphi^{-1}) \circ (\hat \varphi \circ \varphi_2^{-1}) \,,
    \end{equation*}
    where $\varphi_1 \circ \hat \varphi^{-1}$ and $ \hat \varphi \circ \varphi_2^{-1}$ are $W^{k+1, q}$-regular due to (\textit{ii}). By \cref{lemma_Adams}-(\textit{i}), we obtain that $x^i(y) \in W^{k+1, q}_{loc}(\varphi_2(V_1 \cap V_2))$. Now, writing the geometric equation satisfied by $\{x^i\}_{i=1}^n$ in terms of the (harmonic) coordinates $\{y^i\}_{i=1}^n$, we obtain
    \begin{equation*}
        \Delta_g x^l(y) = g^{ij}(y) \frac{\partial^2x^l}{\partial y^i \partial y^j}(y) = 0 \qquad \text{in} \quad \varphi_2(V_1 \cap V_2)
    \end{equation*}
    for each $l = 1, \ldots, n$. Differentiating twice in directions $y^k$ and $y^m$, we obtain
    \begin{equation*}
        g^{ij} \frac{\partial^2}{\partial y^i \partial y^j}\left(\frac{\partial^2 x^l}{\partial y^k\partial y^m}\right) = - \frac{\partial^2 g^{ij}}{\partial y^k\partial y^m} \frac{\partial^2 x^l}{\partial y^i \partial y^j} - \frac{\partial g^{ij}}{\partial y^k} \frac{\partial^3 x^l}{\partial y^i \partial y^j \partial y^m} - \frac{\partial g^{ij}}{\partial y^m} \frac{\partial^3 x^l}{\partial y^i \partial y^j \partial y^k} \,,
    \end{equation*}
    where the right-hand side is in $W^{k-2, q}_{loc}(\varphi_2(V_1 \cap V_2)) \otimes W^{k-1, q}_{loc}(\varphi_2(V_1 \cap V_2)) \hookrightarrow W^{k-2, q}_{loc}(\varphi_2(V_1 \cap V_2))$ for $q > n$. Applying \cref{thm: Laplacian_regularity}-(\textit{ii}), we obtain $\tfrac{\partial^2 x^l}{\partial y^k\partial y^m}(y) \in W^{k, q}_{loc}(\varphi_2(V_1 \cap V_2))$ and consequently $x^l(y) \in W^{k+2, q}_{loc}(\varphi_2(V_1 \cap V_2))$, as desired. Since the choice of harmonic charts was arbitrary, this shows that $\mathcal{A}_{H}$ forms a $W^{k+2, q}$ atlas.

    Finally, were the components $g^{ij}(y) \in W^{l,q}(\varphi_2(V_1 \cap V_2))$ for $l > k$, one could bootstrap the regularity of $x^l(y)$ in the above equation via \cref{thm: Laplacian_regularity}-(\textit{ii}) to $W^{l+2,q}(\varphi_2(V_1 \cap V_2))$.
\end{proof}

\vspace{0.1cm}

\section[{\textbf{Regularity of metric structures}}]{Regularity of metric structures}
\label{section: Regularity of metric structures}

In this section, we prove \cref{thm: main Ricci}. The local theory for $C^{k,\alpha}$ metrics with $k \geq 2$ was done independently by Sabitov--Shefel \cite{Sabitov-Shefel} and DeTurck--Kazdan \cite{DeTurck_Kazdan} in the late seventies. M. Taylor later extended these results to $C^0 \cap W^{1,2}$ metrics \cite[§14, Corollary 12 B.5]{Taylor3}. The key observation is that the metric components, written in harmonic coordinates $\{y^i\}_{i=1}^n$, satisfy the semi-linear elliptic partial differential equation
\begin{equation}
    \label{eq: Ricci harmonic PDE}
    g^{pq}(y)\frac{\partial^2 g_{ij}}{\partial y^p \partial y^q}(y) = -2\R_{ij}(y) +  \,Q_{ij}\bigl(g(y), \partial g(y)\bigr) \,,
\end{equation}
where $Q_{ij}$ is quadratic in $\partial g$. This allowed them to bootstrap the regularity of the components in harmonic coordinates $g_{ij}(y)$ using Schauder theory, 
provided that $\R_{ij}(y)$ is more regular than expected.

The following result is a global version of \cite[Theorem 4.5]{DeTurck_Kazdan} and \cite[Remark 3]{Sabitov-Shefel} for $W^{k,q}$ metrics.

\thmD*
\begin{proof}
    Note that if $l < k-1$, the statement trivially holds with $\Phi = id_M$, so let us assume that $k-1 \leq l$. 
     
    \textit{Step 1.} \, By \cref{prop: harmonic atlas} there exists a finite atlas $\mathcal A_H$ of $M$ consisting of harmonic charts, which is $W^{k+1,q}$-compatible with the $C^\infty$ differential structure of $M$. In any such harmonic chart $(V,\varphi) \in \mathcal A_H$, with coordinates $\{y^i\}_{i=1}^n$, the metric components satisfy \eqref{eq: Ricci harmonic PDE}, where $\R_{ij} \in W^{l,q}_{loc}(\varphi(V))$ and
    \begin{equation*}
        Q_{ij} \in W^{k-1,q}_{loc}(\varphi(V))\otimes W^{k-1,q}_{loc}(\varphi(V)) \hookrightarrow W^{k-1,q}_{loc}(\varphi(V))
    \end{equation*}
    by hypothesis and \cref{lemma_Adams}-(\emph{ii}). Differentiating \eqref{eq: Ricci harmonic PDE} in direction $y^s$ and rearranging, we obtain
    \begin{equation}
        \label{eq: differentiation of harmonic ricci}
        g^{pq}\partial_p\partial_q \bigl(\partial_sg_{ij}\bigl) = -2\,\partial_s\R_{ij} + \partial_sQ_{ij} - \partial_sg^{pq} \, \partial_p\partial_qg_{ij} \,.
    \end{equation}
    Observe that $\partial_s\R_{ij} \in W^{l-1,q}_{loc}(\varphi(V))$, $\partial_sQ_{ij} \in W^{k-2,q}_{loc}(\varphi(V))$ and 
    \begin{equation*}
        \partial_sg^{pq} \, \partial_p\partial_qg_{ij} \in W^{k-1,q}_{loc}(\varphi(V)) \otimes W^{k-2,q}_{loc}(\varphi(V)) \,.
    \end{equation*}
    Due to the multiplication property
    \begin{equation*}
        W^{k-1,q}_{loc}(\varphi(V)) \otimes W^{k-2,q}_{loc}(\varphi(V)) \hookrightarrow W^{k-2,q}_{loc}(\varphi(V)) \,,
    \end{equation*}
    the right-hand-side of \eqref{eq: differentiation of harmonic ricci} is in $W^{k-2,q}_{loc}(\varphi(V))$. Since $\partial_sg_{ij} \in W^{k-1,q}_{loc}(\varphi(V))$, \cref{thm: Laplacian_regularity}-(\textit{ii}) implies that $\partial_sg_{ij} \in W^{k,q}_{loc}(\varphi(V))$ and in turn $g_{ij} \in W^{k+1,q}_{loc}(\varphi(V))$. Furthermore, if $l = k$, we can deduce that the right-hand-side of \eqref{eq: differentiation of harmonic ricci} is actually
    \\

    \noindent
    in $W^{k-1,q}_{loc}(\varphi(V))$ and applying \cref{thm: Laplacian_regularity}-(\textit{ii}) once more, we promote $g_{ij} \in W^{k+2,q}_{loc}(\varphi(V))$. In any case, we have shown that $g_{ij} \in W^{l+2,q}_{loc}(\varphi(V))$.

\vspace{0.1cm}
    \textit{Step 2.} \, In light of \cref{prop: harmonic atlas}-(\textit{iii}), we deduce that the harmonic charts in $\mathcal A_H$ are $W^{l+4,q}$ compatible to each other, and by Sobolev embedding, they form a $C^{l+3}$ differential structure on $M$. On the other hand, it follows from \cref{thm: Withney} that there exists a $C^\infty$ differential structure on $M$ which is $C^{l+3}$ compatible to $\mathcal A_H$. We shall denote by $M'$ the manifold $M$ endowed with this new $C^\infty$ differential structure; we remark that due to the regularity of harmonic coordinates in the original charts of $M$, these two $C^\infty$ differential structures are only $W^{k+1,q}$ compatible. Combining that $g$ is $W^{l+2,q}$ regular in harmonic coordinates with \cref{lemma_Adams}-(\textit{i}), we get that $g \in W^{l+2,q}(M')$. Moreover, \cref{thm: Withney} implies that there exists a $C^\infty$ diffeomorphism $\Phi : M \to M'$. Consequently, $\Phi^*g \in W^{l+2,q}(M)$. Note, however, that $\Phi \in W^{k+1,q}(M,M)$.
\end{proof}

\vspace{0.3cm}
\begin{rmk}
    \label{rmk: improvement obstruction}
     Note that even if $\Ric_g \in C^\infty(M)$, due to the $W^{k+1,q}$ compatibility of any harmonic chart $(V,\varphi)$ with the $C^\infty$ differential structure of $M$, we only have that $\R_{ij} \in W^{k,q}(\varphi(V))$ in harmonic coordinates. Consequently, the regularity of $g$ in harmonic coordinates, and thereby of $\Phi^*g$, can not be improved beyond $W^{k+2,q}$. On the other hand, if we knew that $\R_{ij} \in C^{\infty}(\varphi(V))$ in harmonic coordinates, then we would conclude that $\Phi^*g \in C^\infty(M)$. This is consistent with the obstruction in \cite[Theorem 4.5]{DeTurck_Kazdan}.
\end{rmk}

\begin{rmk}
     The regularity theory developed in \cite[Section 3]{AvalosCogoRoyo} allows for a more general version of \cref{thm: main Ricci}: the hypothesis on the Ricci tensor can be replaced by $\Ric_g \in W^{l,p}$ with $1 < p \leq q$, resulting in $\Phi^* g \in W^{l+3,p}(M)$. The proof is considerably longer, but entails no new ideas.
\end{rmk}

\vspace{0.5cm}

\section[{\textbf{Regularity of conformal structures}}]{Regularity of conformal structures}
\label{section: Regularity of conformal structures}

This section is devoted to the proof of \cref{thm: main}. First, we recall that the \emph{Cotton tensor}
\begin{equation}
    \label{eq: Cotton def}
    \C_{ijk} \coloneqq \nabla_k  \R_{ij} - \nabla_j \R_{ik} + \frac{1}{4} \bigl( \nabla_j \R g_{ik} - \nabla_k \R g_{ij}\bigr) 
\end{equation}
of a smooth Riemannian metric $g$ is conformally invariant in dimension $n = 3$, see \cite{Cotton}.  Taking a divergence and using the Ricci and contracted Bianchi identities, we compute\footnote{We use the convention $\nabla_i\nabla_jX^k - \nabla_j\nabla_iX^k = \R_{ij\;\;l}^{\;\;\;k}X^l$ and $\R_{ij} = g^{pq}\R_{ipjq}$ for curvature.}
\begin{align*}
    \nabla^k \C_{ijk} & = \Delta \R_{ij} - \nabla^k \nabla_j \R_{ik} + \frac{1}{4} \bigl( \nabla^{k} \nabla_j \R \, g_{ik} - \Delta \R \, g_{ij} \bigr) 
    \\
    & = \Delta \R_{ij} - \nabla_j \nabla^k \R_{ki} + \R_{j\,\;k}^{\;kl} \R_{li} + \R_{j\,\;i}^{\;kl} \R_{kl} + \frac{1}{4} \bigl( \nabla_i \nabla_j \R - \Delta \R \, g_{ij} \bigr) 
    \\
    & = \Delta \R_{ij} + \R_i^{\;l}\R_{jl} + \R_{j\,\;i}^{\;kl} \R_{kl} - \frac{1}{4} \bigl(\nabla_i \nabla_j \R + \Delta \R \, g_{ij}\bigr) \,.
\end{align*}
Thus, we find an elliptic equation for the Ricci tensor of the form
\begin{equation}
    \label{eq: Ricci PDE full}
    \Delta_g \Ric_g = \Div_g\C_g + \Rm_g * \Ric_g + \nabla^2 \R_g \,.
\end{equation}
In fact, we claim that \eqref{eq: Ricci PDE full} holds for $g \in W^{2,q}(M)$ metrics, provided that $q > \max\{2,\tfrac{n}{2}\}$.
\begin{lemma}
    \label{lemma: Cotton equation lowreg}
    Let $\Omega \subset \nR^n$ be an open, bounded domain with smooth boundary and $g_{ij} \in W^{2,q}(\Omega)$ a Riemannian metric with $q > \max\{2,\tfrac{n}{2}\}$. Then, 
    \begin{equation*}
        \Delta \R_{ij} = \nabla^k\C_{ijk} - \R_i^{\;l}\R_{jl} - \R_{j\,\;i}^{\;kl} \R_{kl} + \frac{1}{4} \bigl(\nabla_i \nabla_j \R + \Delta \R \, g_{ij}\bigr)
    \end{equation*}
    holds as an equation in $W^{-2,q}(\Omega)$.
\end{lemma}
\begin{proof}
    By the above computation, all we need to check is that the Ricci and contracted Bianchi identities hold for $g_{ij}$. First, observe that the covariant derivative (acting on tensor fields) extends to a bounded linear map
    \begin{equation}
        \label{eq: nabla Lq to W-1q}
        \nabla : L^q(\Omega) \to W^{-1,q}(\Omega)
    \end{equation}
    provided that $q \geq \tfrac{n}{2}$. This follows from the expression $\nabla\mathbf u = \partial\mathbf u + \Gamma*\mathbf u$ of its definition and the Sobolev multiplication
    \begin{equation*}
        W^{1,q}(\Omega) \otimes L^q(\Omega) \hookrightarrow W^{-1,q}(\Omega)
    \end{equation*}
    for $q \geq \tfrac{n}{2}$. Now consider a smooth metric $\tilde g_{ij}$ in $\Omega$ and compute
    \begin{align*}
        \|\nabla_k\R_{ij} - \tilde\nabla_k\tilde\R_{ij}\|_{W^{-1,q}(\Omega)} &\leq \|\nabla_k(\R_{ij} - \tilde\R_{ij})\|_{W^{-1,q}(\Omega)} + \|(\nabla_k - \tilde\nabla_k)\tilde\R_{ij}\|_{W^{-1,q}(\Omega)}
        \\
        &\lesssim \|\R_{ij} - \tilde\R_{ij}\|_{L^{q}(\Omega)} + \|(\nabla_k - \tilde\nabla_k)\tilde\R_{ij}\|_{W^{1,q}(\Omega)} \,,
    \end{align*}
    where we have used the boundedness of \eqref{eq: nabla Lq to W-1q} and the embedding $W^{1,q}(\Omega) \hookrightarrow W^{-1,q}(\Omega)$. It follows that if $\tilde g_{ij} \to g_{ij}$ in $W^{2,q}(\Omega)$, then $\tilde \R_{ij} \to \R_{ij}$ in $L^{q}(\Omega)$ and $\tilde \Gamma^k_{ij} \to \Gamma^k_{ij}$ in $W^{1,q}(\Omega)$, and consequently $\tilde\nabla_k\tilde\R_{ij} \to \nabla_k\R_{ij}$ in $W^{-1,q}(\Omega)$. We conclude by smooth approximation that the contracted Bianchi identity
    \begin{equation*}
        \nabla^k\R_{ki} = \frac{1}{2}\nabla_i\R
    \end{equation*}
    holds in $W^{-1,q}(\Omega)$. Similarly, one can show that $\nabla^2 : L^q(\Omega) \to W^{-2,q}(\Omega)$ is a bounded linear map provided that $q \geq \max\{2,\tfrac{n}{2}\}$ and that the Ricci identity
    \begin{equation*}
        (\nabla_k\nabla_l - \nabla_k\nabla_k)\R_{ij} = -\R_{kl\;\;i}^{\;\;\;p}\R_{pj} - \R_{kl\;\;j}^{\;\;\;p}\R_{ip}
    \end{equation*}
    holds in $W^{-2,q}(\Omega)$ (we refer the reader to \cite[Proposition 4.1]{avalos2024sobolev} for the details). This concludes the proof.
\end{proof}
Combining \eqref{eq: Ricci PDE full} with the techniques in the proof of \cref{thm: main Ricci}, we establish the following:
\thmCottonconstantscalar*
\begin{proof}
    If $l < k-2$, the statement is trivially true with $\Phi = id_M$, so let us assume that $k-2 \leq l$. By \cref{prop: harmonic atlas} there exists a finite atlas $\mathcal A_H$ of $M$ consisting of harmonic charts, which is $W^{k+1,q}$-compatible with the $C^\infty$ differential structure of $M$. The geometric equation \eqref{eq: Ricci PDE full}, together with the constant scalar curvature assumption, implies that in any such harmonic chart $(V,\varphi) \in \mathcal A_H$, with coordinates $\{y^i\}_{i=1}^n$, the components of the Ricci tensor satisfy the elliptic system
    \begin{equation}
        \label{eq: Ricci PDE local}
        \Delta_{g}\bigl(\,\R_{ij}\bigr) = g^{kl}\partial_k\partial_l \R_{ij} = \nabla^k\C_{ijk} + J_{ij} + P_{ij} + Z_{ij} \,,
    \end{equation}
    where
    \begin{equation*}
        J = \Gamma*\Gamma*\Ric\,, \qquad \quad  P = \Ric*\Rm + \partial \Gamma*\Ric \,, \qquad \quad Z = \Gamma*\partial\Ric \,.
    \end{equation*}
    By the hypotheses on the regularity of the tensor fields and \cref{lemma_Adams}-(\textit{ii}), we see that $\nabla^k\C_{ijk} \in W^{l-1,q}_{loc}(\varphi(V))$, while
    \begin{align*}
        J_{ij} \in W^{k-1, q}_{loc}(\varphi(V)) \otimes W^{k-2, q}_{loc}(\varphi(V)) & \hookrightarrow W^{k-2, q}_{loc}(\varphi(V)) 
        \\
        P_{ij} \in W^{k-2, q}_{loc}(\varphi(V)) \otimes W^{k-2, q}_{loc}(\varphi(V)) & \hookrightarrow W^{k-2, q/2}_{loc}(\varphi(V))
        \\
        Z_{ij} \in W^{k-1,q}_{loc}(\varphi(V)) \otimes W^{k-3,q}_{loc}(\varphi(V)) & \hookrightarrow W^{k-3,q}_{loc}(\varphi(V)) \,.
    \end{align*}
     Observe that $W^{k-2, q/2}_{loc}(\varphi(V)) \hookrightarrow W^{k-3, q}_{loc}(\varphi(V))$ for $k \geq 3$ and $q > \tfrac{n}{2}$, whereas $L^{q/2}_{loc}(\varphi(V)) \hookrightarrow W^{-1, q}_{loc}(\varphi(V))$ is ensured by $W^{1, q'}_{loc}(\varphi(V)) \hookrightarrow L^{(q/2)'}_{loc}(\varphi(V))$ for $q > n$. We deduce that $\Delta_{g}\bigl(\,\R_{ij}\bigr) \in W^{k-3,q}_{loc}(\varphi(V))$.
    Moreover, by hypothesis and \cref{lemma_Adams}-(\textit{ii}) it holds that $\R_{ij} \in W^{k-2,q}_{loc}(\varphi(V))$. Applying \cref{thm: Laplacian_regularity}-(\textit{i}) for the case $k = 2$ and \cref{thm: Laplacian_regularity}-(\textit{ii}) for $k \geq 3$, we obtain the improved regularity of the Ricci tensor $\R_{ij} \in W^{k-1,q}_{loc}(\varphi(V))$. Proceeding exactly as in \emph{Step 1} of the proof of \cref{thm: main Ricci}, we obtain the improved regularity of the Yamabe metric $g_{ij} \in W^{k+1, q}_{loc}(\varphi(V))$. Now, if $k-1 \leq l$, this regularity of $g_{ij}$ gives an improved regularity  of $J_{ij}, P_{ij}, Z_{ij} \in W^{k-2, q}_{loc}(\varphi(V))$, so that $\Delta_{\hat{g}}\bigl(\,\R_{ij}\bigr) \in W^{k-2,q}_{loc}(\varphi(V))$ and $\R_{ij} \in W^{k-1,q}_{loc}(\varphi(V))$.
    Applying \cref{thm: Laplacian_regularity}-(\textit{ii}), we obtain  $\R_{ij} \in W^{k,q}_{loc}(\varphi(V))$ and proceeding as in \emph{Step 1} of the proof of \cref{thm: main Ricci}, we obtain $g_{ij} \in W^{k+2,q}_{loc}(\varphi(V))$. Finally, if $l = k$, the iteration of the same argument gives $g_{ij} \in W^{k+3,q}_{loc}(\varphi(V))$. This proves, in general, that $g_{ij} \in W^{l+3,q}_{loc}(\varphi(V))$ in any harmonic chart.

    Proceeding as in \emph{Step 2} of the proof of \cref{thm: main Ricci}, we conclude that $\Phi^*g \in W^{l+3,q}(M)$ for some $W^{k+1,q}$ diffeomorphism $\Phi: M \to M$, as desired.
\end{proof}

\vspace{0.2cm}
\begin{rmk}
    \label{rmk: Cotton improv obstruction}
    As in \cref{rmk: improvement obstruction}, the regularity of $\Phi^*g$ can not generally be improved, even if $\C_{g} \in C^\infty(M)$, unless we already knew that $\C_{ijk} \in C^\infty(\varphi(V))$ in every harmonic chart $(V,\varphi)$.
\end{rmk}
\begin{rmk}
    The regularity theory developed in \cite[Section 3]{AvalosCogoRoyo} allows for a more general version of \cref{thm:Cottonconstantscalar} (and \cref{thm: main}): the hypothesis on the Cotton tensor can be replaced by $\C_g \in W^{l,p}$ with $1 < p \leq q$, resulting in $\Phi^*g \in W^{l+3,p}(M)$. The proof is considerably longer but entails no new ideas.
\end{rmk}
Now, we see that \cref{thm: main} reduces to finding a metric of constant scalar curvature in every class $\llbracket\,g\,\rrbracket_{W^{k,q}}$. In other words, solving the Yamabe problem for this class of Riemannian metrics. This was solved by the authors in a previous paper, under an orientability assumption:
\begin{theorem}[\cite{AvalosCogoRoyo}, Theorem A]
    \label{thm: yamabe problem}
    Let $M$ be an orientable, smooth closed $3$-manifold and consider a Riemannian metric $g \in W^{k,q}(M)$ for $k \geq 2$ and $q > 3$. Then, there exists a positive function $u \in W^{k,q}(M)$ such that $u^4g$ has constant scalar curvature.
\end{theorem}
Finally, we can prove our main theorem:
\thmA*
\begin{proof}[Proof of \cref{thm: main}]
    By \cref{thm: yamabe problem}, there exists some positive function $u \in W^{k,q}(M)$ such that $\hat g \coloneqq u^4g$ has constant scalar curvature. Since the Cotton tensor is conformally invariant, there holds $\C_{\hat g} \in W^{l,q}(M)$ by hypothesis. \cref{thm:Cottonconstantscalar} then implies the assertion.
\end{proof}

\vspace{0.5cm}
\section[{\textbf{Applications}}]{Applications}
\label{section: Applications}

\subsection{Conformal flatness}
\label{subsection: Conformal flatness}
As a first application of \cref{thm: main}, we address the conformal counterpart of a result by M. Taylor. In \cite[Proposition 3.2]{Taylor_ConfFlat}, it is proven that if a Riemannian metric $g_{ij} \in C^{0,\alpha}(\Omega)$ satisfies $\R_{ijkl} \equiv 0$ in $\Omega \subset \nR^n$, then there exists a $C^{1,\alpha}$ local isometry from $(\Omega, g_{ij})$ to an open domain in $\nR^n$. This is a generalization of a classical result to Hölder continuous metrics. The main point of his proof is to avoid reproducing in lower regularity the lengthy proof for smooth metrics using Frobenius theorem (see e.g. \cite[Chapter 27]{EisenhartBook}), and showing instead that the metric is smooth in harmonic coordinates and applying the classical theorem.

We now use \cref{thm: main} to analogously extend to Sobolev metrics the following well-known result\footnote{We point out that on closed $3$-manifolds the condition $\nabla\C_g \equiv 0$ is equivalent to $\C_g \equiv 0$, due to the recent result by I. Terek \cite{Terek}.}:
\begin{theorem}[\cite{EisenhartBook}, Chapter 28]
    \label{thm: Cotton zero equiv loc conf flat}
    Let $(M,g)$ be a smooth, closed Riemannian $3$-manifold. Suppose that $\C_g \equiv 0$. Then $(M,g)$ is locally conformally flat.
\end{theorem}
Instead of reproducing the original proof in lower regularity, we employ our \cref{thm: main} to find a smooth conformal metric and then apply \cref{thm: Cotton zero equiv loc conf flat}. In particular, we show that conformal classes $\llbracket\,g\,\rrbracket_{W^{k,q}}$ with vanishing Cotton tensor always admit a $C^\infty(M)$ representative, in striking contrast to the obstruction highlighted in \cref{rmk: Cotton improv obstruction}. This occurs because $g$ satisfies a conformally invariant equation.

\begin{corA1}[\textbf{Regularity of Cotton flat metrics}]
     Let $M$ be an orientable, smooth, closed $3$-manifold and consider a Riemannian metric $g \in W^{k,q}(M)$ with $k \geq 2$ and $q > 3$. Suppose that $\C_g \equiv 0$. Then there exists a metric $\tilde g \in \llbracket\,g\,\rrbracket_{W^{k,q}}$ such that $\tilde g \in C^\infty(M)$ and is locally conformal flat. In particular, there exists an open neighborhood $U$ around any $p \in M$ and a $W^{k+1,q}$ conformal diffeomorphism
    \begin{equation*}
        \psi :  (\Omega, \delta) \to (U,g)  \,,
    \end{equation*}
    where $\Omega \subset \nR^n$ and $\delta$ is the Euclidean metric.
\end{corA1}
\begin{proof}
    Let $\hat g \in W^{k,q}(M)$ be a constant scalar curvature metric conformal to $g$, which exists by \cref{thm: yamabe problem}. Due to the conformal invariance of the Cotton tensor, $\C_{\hat g} \equiv 0$. Thus, in the right-hand side of the equation \eqref{eq: Ricci PDE local}, only $J_{ij}, P_{ij}, Z_{ij}$ appear, involving the Ricci tensor and the Christoffel symbols and their derivatives. By iterating the local arguments in the proof of \cref{thm: main}, that is, by bootstrapping with \cref{thm: Laplacian_regularity} and \emph{Step 1} of the proof \cref{thm: main Ricci}, one obtains that $\hat g_{ij} \in C^{\infty}(\varphi(V))$ on any harmonic chart $(V, \varphi)$. Moreover, by \cref{prop: harmonic atlas}-(\textit{iii}), the $\hat g$-harmonic atlas $\mathcal{A}_H$ is itself $C^{\infty}$. Denoting by $M'$ the manifold $M$ endowed with the $C^{\infty}$ atlas $\mathcal{A}_H$, \cref{thm: Withney} implies that there exists a $C^\infty$ diffeomorphism $\Phi : M \to M'$. 
    We conclude that $\Tilde{g} \coloneqq \Phi^*\hat{g} = \Phi^*(u^4 g) \in C^{\infty}(M)$ and $\C_{\Tilde{g}} \equiv 0$ by its conformal invariance. Hence, $\Tilde{g}$ is locally conformally flat by \cref{thm: Cotton zero equiv loc conf flat}, that is, there exists an open neighborhood $\Tilde{U}$ around any $\Tilde{p} \in M$, a function  $\Tilde{u} \in C^{\infty}(\Tilde{U})$ and a $C^{\infty}$ diffeomorphism
    \begin{equation*}
        \Tilde{\psi} :  (\Omega, \delta) \to (\Tilde{U},\Tilde{g})  \,,
    \end{equation*}
    where $\Omega \subset \nR^n$ and $\delta$ is the Euclidean metric, such that $\delta = \Tilde{\psi}^* (\Tilde{u}^4 \Tilde{g})$.

    Finally, we observe that $\Phi \in W^{k+1,q}(M,M)$ by the compatibility of $\mathcal{A}_H$ with the original $C^{\infty}$ differential structure due to \cref{prop: harmonic atlas}-(\textit{ii}). Now, for any $p \in M$, we can consider $\Phi^{-1}(p) = \Tilde{p}$ with an open set $\Tilde{U}$ and a diffeomorphism $\Tilde{\psi}$, as above. Then we define the diffeomorphism
    \begin{equation*}
        \psi :  (\Omega, \delta) \to (U,g)  \,, \qquad \quad U \coloneqq \Phi(\Tilde{U}), \quad \psi \coloneqq \Phi|_{\Tilde{U}} \circ \Tilde{\psi}\,,
    \end{equation*}
    which is $W^{k+1, q}$-regular, due to \cref{lemma_Adams}-(\textit{i}) and satisfies
    \begin{equation*}
        \delta = \Tilde{\psi}^* (\Tilde{u}^4 \Tilde{g}) = \Tilde{\psi}^* \left(\Tilde{u}^4\Phi^*(u^4 g)\right) = (\Tilde{u} \circ \Tilde{\psi})^4 (u \circ \Phi \circ \Tilde{\psi})^4 \Tilde{\psi}^* \Phi^* g = (\Tilde{u} \circ \Tilde{\psi})^4 (u \circ \psi)^4 \,\psi^*  g\,.
    \end{equation*}
    This completes the proof.
\end{proof}

\subsection{Static vacuum systems}
\label{Static vacuum systems}
As a second application, we address the regularity of \emph{static vacuum systems}. We highlight that, as well as in the previous application, a geometric partial differential equation enables an upgrade of the regularity of the metric up to $C^{\infty}$, avoiding the obstruction underlined in \cref{rmk: Cotton improv obstruction} and \cref{rmk: improvement obstruction}.

A \emph{static vacuum system} is a tuple $(M, g, f)$, where $(M, g)$ is a Riemmannian $n$-manifold and $f : M \to \nR$ is a positive function, called \emph{static potential}, satisfying the \emph{static vacuum equations}
\begin{equation}\label{Eq: static vacuum}
    \begin{cases}
        \Ric_g =  \frac{\nabla^2 f}{f} + \frac{2\Lambda}{n-1} g \\ 
        \Delta_g f =  - \frac{2\Lambda}{n-1} f
    \end{cases}
\end{equation}
for some constant $\Lambda \in \nR$. From a physical point of view, a static vacuum system $(M, g, f)$ is equivalent to considering a static Lorentzian manifold
\begin{equation*}
    (\mathfrak{L}, \mathfrak{g}) = \bigl(M \times \nR, -f^2 dt^2 + g\bigr)
\end{equation*}
which satisfies the \emph{Einstein vacuum equations}  
$$\mathfrak{Ric} - \frac{\mathfrak{R}}{2} \mathfrak{g} + \Lambda \mathfrak{g}  = 0.$$ 
The quantity $\Lambda$ is the so-called \emph{cosmological constant} and tracing the first equation in \eqref{Eq: static vacuum} we see that 
\begin{equation}
    \label{eq: Static constant scalar}
    \R_g = 2\Lambda.
\end{equation}
In particular, $g$ is a Yamabe metric. The system \eqref{Eq: static vacuum} also shows up in the study of the adjoint of the linearized scalar curvature operator, see \cite{FischerMarsen_static,Kobayashi_static} for example. If $M$ is closed, integrating the second equation of \eqref{Eq: static vacuum} we deduce that $\Lambda = 0$ and $f$ is a constant function, which in turn implies that $(M,g)$ is Ricci flat. We thus consider more generally static vacuum systems that need not be compact and we do not require any structure at infinity.

First, we note that the Cotton tensor of static vacuum systems is more regular than the expected one.
\begin{lemma}\label{lemma: Cotton_static}
    Let $\Omega \subset \nR^n$ be an open bounded domain and $g_{ij} \in W^{k,q}_{loc}(\Omega)$ and $f \in W^{k,q}_{loc}(\Omega)$ a Riemannian metric and a positive function with $k \geq 2$ and $q > n$ satisfying \eqref{Eq: static vacuum}. Then, $\C_{ijk} \in W^{k-2,q}_{loc}(\Omega)$.
\end{lemma}
\begin{proof}
    Using the first equation of \eqref{Eq: static vacuum}, \eqref{eq: Static constant scalar}
    and the definition of the Riemann tensor, we compute the Cotton tensor \eqref{eq: Cotton def}
    \begin{align*}
       \C_{ijk} & = \nabla_k  \left( \frac{\nabla_i \nabla_j f}{f} +  \Lambda g_{ij}\right) - \nabla_j \left( \frac{\nabla_i \nabla_k f}{f} +  \Lambda g_{ik}\right) 
       \\
       & = f^{-1} \left( \nabla_k \nabla_j - \nabla_j \nabla_k \right) \nabla_i f - f^{-2} \left( \nabla_k f \cdot \nabla_j  \nabla_i f - \nabla_j f \cdot \nabla_k  \nabla_i f \right)  
       \\
       & = f^{-1} \R_{kj \;\; i}^{\;\;\;\;l} \nabla_l f - f^{-2} \left( \nabla_k f \cdot \nabla_j  \nabla_i f - \nabla_j f \cdot \nabla_k  \nabla_i f \right) \,.
    \end{align*}
    Since the $\R_{ijkl} \in W^{k-2,q}_{loc}(\Omega)$, the Sobolev multiplication 
    \begin{equation*}
        W^{k,q}_{loc}(\Omega) \times W^{k-1,q}_{loc}(\Omega) \times W^{k-2,q}_{loc}(\Omega) \hookrightarrow W^{k-2,q}_{loc}(\Omega)
    \end{equation*}
    ensures that $\C_{ijk} \in W^{k-2}_{loc}(\Omega)$.
\end{proof}
In \cite{Corvino_static}, J. Corvino showed that the metric components of a $C^2$ static system are analytic in local harmonic coordinates. In the following, we establish a global version thereof for metrics of lower regularity:
\begin{corB1}[\textbf{Regularity of static systems}]
     Let $(M, g, f)$ be a static system with $(g,f) \in W^{k,q}_{loc}(M) \times W^{k,q}_{loc}(M)$ for some $k \geq 2$ and $q > n$. Then, there exists a $W^{k+1,q}_{loc}$ diffeomorphism $\Phi : M \to M$ such that $(\Phi^*g,\Phi^*f) \in C^{\infty}_{loc}(M) \times C^{\infty}_{loc}(M)$.
\end{corB1}
\begin{proof}
    Fix a harmonic chart $(V,\varphi)$ around any point in $M$ and notice, as in \eqref{eq: Ricci PDE local} and \eqref{eq: Ricci harmonic PDE}, that the components of the Ricci and the metric tensor satisfy the system
    \begin{align}
        \label{eq: bootstrap system 1}
        \Delta_{g}\bigl(\,\R_{ij}\bigr) &= \nabla^k\C_{ijk} + J_{ij} + P_{ij} + Z_{ij}
        \\
        \label{eq: bootstrap system 2}
        \Delta_g\bigl(g_{ij}\bigr) &= -2\R_{ij} +  \,Q_{ij}
    \end{align}
    in $\varphi(V)$, where
    \begin{align*}
        J = \Gamma*\Gamma*\Ric\,, \quad  P = \Ric*\Rm + \partial \Gamma*\Ric \,, \quad Z = \Gamma*\partial\Ric \,, \quad Q = \partial g * \partial g \,.
    \end{align*}
    Again, the first equation is justified by \cref{lemma: Cotton equation lowreg}. By the same arguments as in the proof of \cref{thm:Cottonconstantscalar} we see that $J_{ij},P_{ij},Z_{ij},Q_{ij} \in W^{k-3,q}_{loc}(\varphi(V))$, while \cref{lemma: Cotton_static} implies that $C_{ijk} \in W^{k-2}_{loc}(\varphi(V))$. It is then a consequence of \cref{thm: Laplacian_regularity} applied to \eqref{eq: bootstrap system 1} that $\R_{ij} \in W^{k-1,q}_{loc}(\varphi(V))$. Now, differentiating \eqref{eq: bootstrap system 2} and applying \cref{thm: Laplacian_regularity} like in the proof of \cref{thm: main Ricci} we deduce that $g_{ij} \in W^{k+1,q}_{loc}(\varphi(V))$. Moreover, by the second equation in \eqref{Static vacuum systems} and \cref{thm: Laplacian_regularity}, $f \in W^{k+1,q}_{loc}(\varphi(V))$. In turn, this implies that $J_{ij},P_{ij},Z_{ij},Q_{ij} \in W^{k-2,q}_{loc}(\varphi(V))$ by Sobolev multiplications and $\C_{ijk} \in W^{k-1,q}_{loc}(\varphi(V))$ by \cref{lemma: Cotton_static}, so we may bootstrap the argument to conclude that $g_{ij} \in C^\infty_{loc}(\varphi(V))$ and $f \in C^{\infty}_{loc}(\varphi(V))$.

    We may now consider one such harmonic neighborhood $(V_p,\varphi_p)$ around each point $p \in M$ and appeal to \cref{prop: harmonic atlas}-(\textit{iii}) to conclude that $\mathcal A = \{(V_p,\varphi_p)\}_{p \in M}$ forms a $C^\infty$ atlas on $M$. We name $M'$ the $C^1$ manifold $M$ endowed with the $C^\infty$ differential structure generated by $\mathcal A$ and remark that $g \in C^\infty_{loc}(M')$. By \cref{thm: Withney}, there exists some $C^\infty$ diffeomorphism $\Phi : M \to M'$. Therefore, $\Phi^*g \in C^\infty_{loc}(M)$ and $\Phi^*f \in C^\infty_{loc}(M)$, as desired. 
\end{proof}
Naturally, the same argument works for Einstein manifolds, that is, manifolds satisfying $\Ric_g = \lambda g$ for some constant $\lambda$ and hence, in particular, the Cotton vanishes. It may be regarded as a global version of \cite[Theorem 5.2]{DeTurck_Kazdan} to Sobolev metrics:
\begin{corB2}[\textbf{Regularity of Einstein metrics}]
    Let $M$ be a smooth $n$-manifold and let $g \in W^{k,q}_{loc}(M)$ be an Einstein metric on $M$ with $k \geq 2$ and $q > \tfrac{n}{2}$. Then there exists a $W^{k+1,q}_{loc}$ diffeomorphism $\Phi: M \to M $ such that $\Phi^* g \in C^{\infty}_{loc}(M)$.
\end{corB2}
We remark that \cref{coroll: regularity static systems} and \cref{coroll: regularity Einstein} could be proven just by using the arguments in \cref{thm: main Ricci} since the underlying PDE involves the Ricci tensor. Nevertheless, while Ricci has a priori only the expected regularity, we identify a geometric quantity, that is the Cotton tensor, carrying the improved regularity globally and in a coordinate-independent way. 
We believe that this observation can be a powerful tool elsewhere.

\vspace{0.5cm}

\printbibliography[]

\vspace{0.8cm}

\end{document}

%% file: Preambolo.tex
\usepackage[utf8]{inputenc}
\usepackage{amstext,amsmath, amssymb, amsthm, mathrsfs, mathtools}
\usepackage[a4paper, total={6.2in, 9.1in}, left={1.3in}, right={1.3in}, top={1.1in},bottom={1.3in}, footskip={.5in}]{geometry}
\usepackage{enumerate}
\usepackage{xcolor}
\usepackage{comment}
\usepackage{hyperref}
\usepackage[capitalise]{cleveref}
\usepackage{etoolbox}
\usepackage{stmaryrd}

\usepackage{setspace}
\allowdisplaybreaks

\makeatletter \def\l@subsection{\@tocline{1}{0pt}{3pc}{2pc}{}} 
\makeatother

\DeclareRobustCommand{\SkipTocEntry}[5]{} 

\usepackage[backend=bibtex]{biblatex}
\newtheorem{theorem}{Theorem}[section]
\newtheorem{lemma}[theorem]{Lemma}
\newtheorem{rmk}[theorem]{Remark}
\newtheorem{proposition}[theorem]{Proposition}

\theoremstyle{definition} 

\numberwithin{equation}{section}

\DeclareMathOperator{\nR}{\mathbb{R}}
\DeclareMathOperator{\nE}{\mathbb{E}}

\DeclareMathOperator{\nN}{\mathbb{N}}

\DeclareMathOperator{\tr}{tr}

\DeclareMathOperator{\Rm}{Rm}
\DeclareMathOperator{\Ric}{Ric}
\DeclareMathOperator{\R}{R}

\DeclareMathOperator{\Div}{div}

\DeclareMathOperator{\C}{C}

\makeatletter
\newtheorem*{rep@theorem}{\rep@title}
\newcommand{\newreptheorem}[2]{%
\newenvironment{rep#1}[1]{%
 \def\rep@title{#2 \ref{##1}}%
 \begin{rep@theorem}}%
 {\end{rep@theorem}}}
\makeatother

\newreptheorem{theorem}{Theorem}

\newreptheorem{corollary}{Corollary}

\usepackage{thmtools}
\usepackage{thm-restate}

\newtheorem*{corA1}{Corollary A.1}
\newtheorem*{corB1}{Corollary B.1}
\newtheorem*{corB2}{Corollary B.2}

\newtheorem{corA}{Corollary}

\newtheorem{corB}{Corollary}

\newreptheorem{corA}{Corollary}
\newreptheorem{corB}{Corollary}